\documentclass[11pt,a4paper,oneside,lualatex]{amsart}
\usepackage[marginparwidth=0pt,margin=25truemm]{geometry} 

\usepackage{mypackage} 
\usepackage{mycommand} 

\DeclareMathOperator{\val}{val}

\usepackage[luatex, pdfencoding=auto,hypertexnames=false]{hyperref}
\hypersetup{colorlinks=false}

\numberwithin{equation}{section} 
\usepackage{mytheoremeng} 

\begin{document}

\title[Extended-cycle integrals]{Extended-cycle integrals of modular functions for badly approximable numbers}

\author[Y. Murakami]{Yuya Murakami}
\address{Mathematical Inst. Tohoku Univ., 6-3, Aoba, Aramaki, Aoba-Ku, Sendai 980-8578, JAPAN}
\email{yuya.murakami.s8@dc.tohoku.ac.jp}

\thanks{The author is supported by JSPS KAKENHI Grant Number JP 20J20308.	}

\date{\today}

\maketitle


\begin{abstract}
	Cycle integrals of modular functions are expected to play a role in real quadratic analogue of singular moduli.
	In this paper, we extend the definition of cycle integrals of modular functions from real quadratic numbers to badly approximable numbers.
	We also give explicit representations of values of extended-cycle integrals for some cases.
\end{abstract}

\tableofcontents


\section{Introduction} \label{sec:intro}


The elliptic modular $ j $-function $ j(z) $ is an $ \SL_2(\Z) $-invariant holomorphic function on the upper half-plane $ \bbH := \{ z=x+y\iu \in \bbC \mid x, y \in \R, y>0 \} $.
It plays an essential role in the complex multiplication theory.
Indeed, its special values at imaginary quadratic points generate the Hilbert class fields of imaginary quadratic fields.
On the other hand, for real quadratic fields, any definitive construction of the Hilbert class fields or ray class fields is still unknown except for a few cases, that is \cite{Shimura}, \cite{Shintani}.
With this motivation, Kaneko~\cite{Kan} introduced  the ``value" of $ j(z) $, written $ \val(w) $ derived from Dedekind's notation, at any real quadratic number $ w $ which is defined by
\begin{equation} \label{eq:def_val}
	\val(w) := \frac{1}{2 \log \veps_{w}}
	\int_{z_0}^{\gamma_w z_0} j(z) \bigg( \dfrac{1}{z - w'} - \dfrac{1}{z - w} \bigg) dz,
\end{equation}
where  $ z_0 \in  \bbH $ is any point, $ w' $ is the non-trivial Galois conjugate of $ w $,
$ \SL_2(\Z)_w $ is the stabilizer of $ w $ in $ \SL_2(\Z) $, and
$ \gamma_w = \begin{pmatrix} * & * \\ c & d \end{pmatrix} $ is the unique element of $ \SL_2(\Z)_w $ such that 
$ \SL_2(\Z)_w = \{ \pm \gamma_w^n \mid n \in \Z \} $ and $ \veps_w := cw + d > 1 $.
The integral is independent of $ z_0 $ since the integrand is holomorphic and invariant under $ \gamma_w $.
We can prove that $  \veps_{w} $ is the minimal unit greater than $ 1 $ with the norm $ 1 $ in the order $ \calO_{w} $ in the real quadratic field $ \Q(w) $ whose discriminant coincides with the discriminant of $ w $.

This value can be written as $ \val(w) = \widetilde{\val}(w) / \widetilde{1}(w)$, where
\[
\widetilde{\val}(w) :=  \int_{\SL_2(\Z)_w \backslash S_{w', w}} j ds, \quad
\widetilde{1}(w) := \int_{\SL_2(\Z)_w \backslash S_{w', w}} ds = 2 \log \veps_{w}
\]
are cycle integrals, $ ds := y^{-1} \sqrt{dx^2 + dy^2} $, and $ S_{w', w} $ is the geodesic in $ \bbH $ from $ w' $ to $ w $.

In the last decade, it turned out that the values of $ \widetilde{\val} $ have interesting properties.
They are related to a mock modular form of weight $ 1/2 $ for $ \Gamma_0(4) $ studied by Duke--Imamoglu--T\'{o}th~\cite{DIT}, which is a real quadratic analogue of Zagier's work on traces of singular moduli in~\cite{Zag}.
They are also related to a locally harmonic Maass form of weight $ 2 $ studied by Matsusaka~\cite[Theorem 1.1]{Matsusaka_Rade}.
Bengoechea--Imamoglu~\cite[Theorem 1.1]{BI} showed some kind of continuity for the values of $ \val $ at Markov quadratics which is conjectured by Kaneko~\cite{Kan}.
The author~\cite[Theorem 1.1]{Mura} generalized it to the values of $ \val $ at real quadratic numbers whose periods of continued fraction expansions have long cyclic parts. It revealed that the continuity of $ \val $ differs from Euclidean topology.
However, the definition of the continuity of $ \val $ is still missing.
We would expect a topological space $ H $ containing all real quadratic numbers and a continuous function $ ? \colon H \to \bbC $ such that the diagram
\begin{equation}
	\begin{tikzcd}
		H \ar[rd, "?"] \\
		\{ \text{real quadratic numbers} \} \ar[r, "\val" '] \ar[u, hook] & \bbC
	\end{tikzcd}
\end{equation}
commutes.

In this paper, we attempt to choose $ H $ as a set of badly approximable numbers and give some partial continuity of $ \val $ on $ H $.
Here, badly approximable numbers are defined as irrational numbers $ x $ whose regular continued fraction expansions
\[
x = [k_1, k_2, k_3, \dots] := k_1 + \cfrac{1}{k_2 + {\cfrac{1}{k_3 + \cfrac{1}{\ddots} } } }, \quad
k_1 \in \Z, \quad
k_2, k_3, \dots \in \Z_{>0}
\]
have bounded coefficients, that is, the set $ \{ k_1, k_2, k_3, \dots \} $ is bounded.
For example, real quadratic irrational numbers are badly approximable since they have periodic continued fraction expansions.

We define $ \val(x) $ and related values for a badly approximable number $ x = [k_1, k_2, k_3, \dots] $ as
\begin{align}
	\val(x) &:= \lim_{
		\substack{ z \in S_{z_0, x} \\
			z \to x}}
	\frac{1}{d_{\mathrm{hyp}}(z_0, z)}
	\int_{S_{z_0, z}} j ds, 
	\label{eq:val_new_intro1} \\
	\widehat{\val}(x) &:=
	\lim_{n \to \infty} \frac{1}{n}
	\int_{z_0}^{\gamma_{n}^{} z_0} j(z) \frac{-1}{z-x} dz, 
	\label{eq:val_new_intro2}\\
	\widehat{1}(x) &:=
	\lim_{n \to \infty} \frac{1}{n}
	\int_{z_0}^{\gamma_{n}^{} z_0} \frac{-1}{z-x} dz,
	\label{eq:val_new_intro3} \\
	\widehat{\veps}_x &:= \lim_{n \to \infty} c_n^{1/n},
	\label{eq:val_new_intro4}
\end{align}
where $ z_0 \in \bbH $, $ S_{z_0, z}  $ is the geodesic in $ \bbH $ from $ z_0 $ to $ z \in \bbH \cup \R $, 
\[
d_{\mathrm{hyp}}(z_0, z) := \int_{S_{z_0, z}} ds, \quad
\gamma_n := 
\begin{pmatrix} k_1 & 1 \\ 1 & 0 \end{pmatrix} 
\begin{pmatrix} k_2 & 1 \\ 1 & 0 \end{pmatrix}
\dots 
\begin{pmatrix} k_{2n} & 1 \\ 1 & 0 \end{pmatrix}
\in \SL_2(\Z),
\]
and $ c_n \in \Z_{>0} $ be the denominator of the rational number $ [k_1, k_2, \cdots k_{2n}] $.
In \cref{sec:fund} later, we will introduce more general representations of these values.

We will prove that these limit values are bounded, but we do not know if the limit values are singleton. Hence, the right-hand sides of the above equations do converge for any badly approximable number.
We can also prove that if $ x $ is a real quadratic number, then the limits $ \val(x), \widehat{\val}(x), \widehat{1}(x) $, and $ \widehat{\veps}_x $ converge, $ \val(x) $ coincides with the values defined in \cref{eq:def_val}, and we can write
\[
\widehat{\val}(x) = \frac{\widetilde{\val}(x)}{r_x}, \quad
\widehat{1}(x) = \frac{\widetilde{1}(x)}{r_x}, \quad
\widehat{\veps}_x = \veps_{x}^{1/r_x},
\]
where $ r_x $ are half of the length of the minimum even period of the continued fraction expansion of $ x $.
Thus, $ \widehat{\veps}_x $ for a badly approximable number $ x $ is an analog of fundamental units of real quadratic orders.

Our purpose in this paper is to give uncountable infinitely many badly approximable numbers $ x $ such that the limits $ \val(x), \widehat{\val}(x), \widehat{1}(x), $ and $ \veps_x $ converge.
We also explicitly represent values of extended $ \val $ function at badly approximable numbers which have infinitely long cyclic parts in their continued fractions. 

To explain the main results, we introduce some notation for even words.
For an even number $ 2r \in 2\Z_{\ge 0} $ and positive integers $ k_1, \dots, k_{2r} $, we call a tuple $ W=(k_1, \dots, k_{2r}) $ an even word.
In the case when $ r = 0 $, we call it the empty word and denote it by $ \emptyset $.
For even words $ W_1 = (k_1^{(1)}, \dots, k_{2r_1}^{(1)}), \dots, W_n = (k_1^{(n)}, \dots, k_{2r_n}^{(n)}) $, define their product
\[
W_1 \cdots W_n :=(k_1^{(1)}, \dots, k_{2r_1}^{(1)}, \dots, k_1^{(n)}, \dots, k_{2r_n}^{(n)}).
\]
For an infinite sequence of even words $ W_1 = (k_1^{(1)}, \dots, k_{2r_1}^{(1)}), \dots, W_n = (k_1^{(n)}, \dots, k_{2r_n}^{(n)}), \dots $, define 
\[
[W_1 \cdots W_n \cdots] :=
[k_1^{(1)}, \dots, k_{2r_1}^{(1)}, \dots, k_1^{(n)}, \dots, k_{2r_n}^{(n)}, \dots].
\]
For a non-empty even word $ W $, define
\begin{align}
	N(W) &:= \max \{ n \in \Z_{>0} \mid \text{ there exists  an even word } W_1 \text{ such that } W = W_1^n \}.
\end{align}

Our first main result is as follows.

\begin{thm} \label{thm:main:realize_lim}
	Let $V_0, \dots, V_k $ be even words, $ W_1, \dots, W_k$ be non-empty even words, and $\{ a_{1, n} \}_{n=1}^{\infty}$, $ \dots $, $ \{ a_{k, n} \}_{n=1}^{\infty}$ be sequences in $ \Z_{\ge 0} $ such that $ a_{i, n} \to \infty $ and $ 2^{-n} a_{i, n} \to 0 $ when $ n $ goes to infinity.
	Let $ w_i := [\overline{W_i}] $,
	$ k' := \{ 0 \le i \le k \mid V_i \neq \emptyset \} $,
	\[
	A_n := k'n + \sum_{1 \le i \le k} \sum_{1 \le j \le n} a_{i, j}, \quad
	a_i := \lim_{n \to \infty} \frac{1}{A_n} \sum_{1 \le j \le n} a_{i, j},
	\]
	and
	\[
	U_n := V_1 W_1^{a_{1, n}} V_2 W_2^{a_{2, n}} \cdots V_k W_k^{a_{k, n}}, \quad
	x := [U_1 U_2 \cdots U_n \cdots ].
	\]
	Then the limits $ \val(x), \widehat{\val}(x), \widehat{1}(x), $ and $ \widehat{\veps}_x $  converge and it holds
	$ \widehat{1}(x) = 2 \log \widehat{\veps}_x > 0 $ and $ \val(x) = \widehat{\val}(x)/\widehat{1}(x) $.
	Moreover, it also holds
	\begin{equation} \label{eq:val_explicit}
		\val(x) =  
		\frac{a_1 N(W_1) \widehat{\val}(w_1) + \dots + a_k N(W_k) \widehat{\val}(w_k)}
		{a_1 N(W_1) \widehat{1}(w_1) + \dots + a_k N(W_k) \widehat{1}(w_k)}.
	\end{equation}
\end{thm}

Here we remark that the author proved that the right-hand side in \cref{eq:val_explicit} equals to the limit of $ \val([\overline{U_n}]) $ in \cite[Theorem 1.1]{Mura}.

Next, we consider the value of $ \val $ function at the Thue--Morse word.
It is a typical word which is repetition-free.
We fix two even words $ V $ and $ W $.
Let $ \{ V, W \}^\omega $ be the set of even words generated by $ V, W $ which are of finite length.
Let $ \{ V, W \}^\omega $ be the set of even words generated by $ V, W $ which are of infinite length.
We define the monoid homomorphism $ h \colon \{ V, W \}^* \to \{ V, W \}^* $ by $ h(V) := VW, h(W) := WV $. 
In this setting, the word
\[
V_h := \lim_{n \to \infty} h^n(V) = VWWVWVVW \cdots \in \{ V, W \}^\omega
\]
is called the Thue--Morse word.
It is known that $ V_h $ is cubefree (for example, see~\cite[Section 3]{ChrisKarhu}), and thus it is not a word which satisfies the conditions in \cref{thm:main:realize_lim}.

The value of $ \val $ function at the Thue--Morse word is calculated as follows.

\begin{thm} \label{thm:Thue--Morse}
	Let $ x := [V_h] $ and $ w_n := [\overline{h^n(V)}] $.
	Then $ \widehat{\val}(x) $ converges if and only if $ \lim_{n \to \infty} \widehat{\val}(w_{2n})/2^{2n} $ converges.
	Moreover, we have
	\[
	\widehat{\val}(x) = \lim_{n \to \infty} \frac{1}{2^{2n}} \widehat{\val}(w_{2n}), \quad
	\widehat{1}(x) = \lim_{n \to \infty} \frac{1}{2^{2n}} \widehat{1}(w_{2n}), \quad
	\val(x) = \lim_{n \to \infty} \frac{1}{2^{2n}} \val(w_{2n}).
	\]
\end{thm}

The definitions of $ \widehat{\val}(x) $ and $ \widehat{1}(x) $ are like the Birkoff average in ergodic theory.
Khinchin-L\'{e}vy's theorem in ergodic theory asserts that for almost all $ x \in [0, 1] $ with respect to the measure
\[
d \mu := \frac{1}{\log 2} \frac{dx}{1+x}
\]
on $ [0, 1] $, the limit defining $ \widehat{\veps}_x $ converges and we have
\[
\log \widehat{\veps}_x
= -2 \int_{0}^{1} \log(x) d \mu
= \frac{\pi^2}{6 \log 2}.
\]

Although ergodic theory makes it possible to study values like the Birkoff average at almost all points, it does not tell us about a value at each point, for example, a badly approximable number.
We can state that this paper studies a range which ergodic theory does not deal with.

This paper will be organized as follows. 
In \cref{sec:real_quad}, we give renewal definitions of $ \val(w), \widehat{\val}(w) $, and $ \widehat{1}(w) $ for real quadratic numbers $ w $, which are suitable to generalize for badly approximable numbers.
In \cref{sec:fund}, we state some fundamental properties of the limits $ \val(x), \widehat{\val}(x), \widehat{1}(x), $ and $ \widehat{\veps}_x $ for each badly approximable number $ x $.
In \cref{sec:def_geodesic}, we describe $ \val(x) $ in terms of the limit value along a geodesic and study its properties.
In \cref{sec:def_word}, we also study basic properties of $ \widehat{\val}(x) $ and $ \widehat{1}(x) $ which are defined using the continued fraction expansion of $ x $.
In \cref{sec:elem_unit}, we study $ \veps_x $ and a relation between it and $ \widehat{1}(x) $.
Finally we prove \cref{thm:main:realize_lim} in \cref{sec:explicit_computation} and \cref{thm:Thue--Morse} in \cref{sec:Thue--Morse}.


\section*{Acknowledgement}


I would like to show my greatest appreciation to Professor Takuya Yamauchi for giving many pieces of advice. 
I am deeply grateful to Dr.~Toshiki Matsusaka for giving many comments. 
I would like to express my gratitude to Professor Shun'ichi Yokoyama and Dr.~Toshihiro Suzuki for giving me constructive comments on computing $ \val(x) $ numerically.
It is a pleasure to extend my thanks to Professor Tatsuya Tate for teaching me ergodic theory.
I also thank Dr.~Daisuke Kazukawa, Dr.~Hiroki Nakajima, and Dr.~Shin'ichiro Kobayashi for teaching me geodesics, hyperbolic geometry, and metric spaces.
I appreciate the technical assistance of Dr.~Naruaki Kato for introducing me to how to write works by using GitHub.



\section{A renewal definition of $ \val $ for real quadratic numbers} \label{sec:real_quad}


In this section, we give renewal definitions of $ \val $ for real quadratic numbers $ w $, which are suitable to generalize for badly approximable numbers.
To begin with, we list some basic notations and facts.


\subsection{Basic notation and properties for geodesics} \label{subsec:basic_geodesics}


Let $ \bbH := \{ z=x+y\iu \in \bbC \mid x, y \in \R, y>0 \} $ be the upper half plane and $ ds := y^{-1} \sqrt{dx^2 + dy^2} $ be its line element.
Here we remark that $ ds $ is invariant under $ \SL_2(\R) $.
We denote by $ \infty := \lim_{t \to + \infty} t \iu $ the point at infinity.
Let $ \bbP^1(\R) := \R \cup \{ \infty \} $.
For two points $ x', x \in \bbH \cup \bbP^1(\R) $, we denote by $ S_{x', x}  $ the geodesic in $ \bbH $ from $ x' $ to $ x $. 
For two points $ z, z' \in \bbH $, define hyperbolic distance between them by
\[
d_{\mathrm{hyp}}(z, z') := \int_{S_{z, z'}} ds.
\]
The following lemma is fundamental.

\begin{lem} \label{lem:distance}
	For two points $ z, z' \in \bbH $, the following statements hold.
	\begin{enumerate}
		\item \label{item:lem:distance1} 
		For any $ \sigma \in \SL_2(\R) $, it holds $ d_{\mathrm{hyp}}(\sigma(z), \sigma(z')) = d_{\mathrm{hyp}}(z, z') $.
		\item \label{item:lem:distance2}
		\textup{(\cite[Theorem 7.2.1 (ii)]{Bea})}
		\[
		\cosh d_{\mathrm{hyp}}(z, z') 
		= 1 + \frac{\abs{z - z'}^2}{2 \Im(z) \Im(z')}.
		\]
	\end{enumerate}
\end{lem}

For $ x', x \in \bbP^1(\R) $, define a $ 1 $-form
\[
\eta_{x', x} (z) := \bigg( \dfrac{1}{z - x'} - \dfrac{1}{z - x} \bigg) dz.
\]
Here, let $ 1/(z - \infty) := 0 $ if $ x = \infty $ or $ x' = \infty $.

By a direct calculation, we have the following lemma.

\begin{lem}[{\cite[Lemma 2.5]{Mura}}] \label{lem:eta}
	For any $ x, x' \in \bbP^1(\R) $ and $ \gamma \in \SL_2(\R) $, we have 
	\[
	\gamma^* \eta_{x', x} = \eta_{\gamma^{-1}x', \gamma^{-1}x}.
	\]
\end{lem}

We need the following lemma to express $ \val(w) $ as a quotient of two cycle integrals.

\begin{lem} \label{lem:eta_on_geod}
	For two points $ x, x' \in \bbP^1(\R) $ with $ x \neq x' $, we have $ ds = \eta_{x', x} $ on $ S_{x', x} $. 
\end{lem}

\begin{proof}
	Let
	\[
	\sigma :=
	\begin{cases}
		\begin{pmatrix} x & x'/(x-x') \\ 1 & 1/(x-x') \end{pmatrix} \quad &\text{if } x \neq \infty,\\
		\begin{pmatrix} 1 & -x \\ 0 & 1 \end{pmatrix} \quad &\text{if } x = \infty
	\end{cases}
	\]
	be a matrix in $ \SL_2(\R) $.
	Since $ \sigma^* \eta_{x', x} = \eta_{0, \infty} $ and $ \sigma^* ds = ds $, it is enough to show when $ x=\infty $ and $ x'=0 $. 
	On the geodesic $ S_{0, \infty} = \{ t \iu \mid t \in \R_{>0} \} $, we have
	$ds = dt/t =  \eta_{0, \infty} $.
\end{proof}


\subsection{Basic notation and properties for real quadratic numbers} \label{subsec:basic_real_quad}


For a real quadratic number $ w $, let $ w' $ be the non-trivial Galois conjugate of $ w $,
$ \SL_2(\Z)_w $ be the stabilizer of $ w $ in $ \SL_2(\Z) $, and
$ \gamma_w = \begin{pmatrix} * & * \\ c & d \end{pmatrix} $ be the unique element of $ \SL_2(\Z)_w $ such that 
$ \SL_2(\Z)_w = \{ \pm \gamma_w^n \mid n \in \Z \} $ and $ \veps_w := cw + d > 1 $.
It holds $ \gamma_w w' = w' $.
Thus, the geodesic $ S_{w', w} $ is invariant under $ \gamma_w $.
It also holds $ \gamma_w^* \eta_{w', w} = \eta_{w', w} $ by \cref{lem:eta}.


\subsection{Basic notation and properties for continued fractions and words} \label{subsec:basic_words}


For the convenience of the reader, we recall notations for even words prepared in \cref{sec:intro}.

Each real number can be represented by the unique continued fraction
\[
[k_1, k_2, k_3, \dots] := k_1 + \cfrac{1}{k_2 + {\cfrac{1}{k_3 + \cfrac{1}{\ddots} } } }
\]
where $ k_1 $ is an integer and $ k_2, k_3, \dots $ are positive integers.
For positive integers $ k_1 ,\dots, k_s $, we set a periodic continued fraction
\[
[k_1, \dots, k_{r}, \overline{k_{r+1}, \dots, k_s}] := [k_1, \dots, k_{r}, k_{r+1}, \dots, k_s, k_{r+1}, \dots, k_s, k_{r+1}, \dots, k_s, \dots].
\]
For a real number, being quadratic (resp.~rational) is equivalent to having a periodic (resp.~finite) continued fraction expansion (\cite[Theorem 1.17]{Aig}).

For an even number $ 2r \in 2\Z_{\ge 0} $ and positive integers $ k_1, \dots, k_{2r} $, we call a tuple $ W=(k_1, \dots, k_{2r}) $ an even word.
In the case when $ r = 0 $, we call it the empty word and denote it by $ \emptyset $.
For even words $ W_1 = (k_1^{(1)}, \dots, k_{2r_1}^{(1)}), \dots, W_n = (k_1^{(n)}, \dots, k_{2r_n}^{(n)}) $, define their product
\[
W_1 \cdots W_n :=(k_1^{(1)}, \dots, k_{2r_1}^{(1)}, \dots, k_1^{(n)}, \dots, k_{2r_n}^{(n)}).
\]
For an infinite sequence of even words $ W_1 = (k_1^{(1)}, \dots, k_{2r_1}^{(1)}), \dots, W_n = (k_1^{(n)}, \dots, k_{2r_n}^{(n)}), \dots $, define 
\[
[W_1 \cdots W_n \cdots] :=
[k_1^{(1)}, \dots, k_{2r_1}^{(1)}, \dots, k_1^{(n)}, \dots, k_{2r_n}^{(n)}, \dots].
\]
For the empty word $ \emptyset $ and a non-empty even word $ W=(k_1, \dots, k_{2r}) $, define
\begin{align}
	\gamma_\emptyset^{} &:= \begin{pmatrix} 1 & 0 \\ 0 & 1 \end{pmatrix} \in \SL_2(\Z), \\
	\gamma_W^{} &:= \begin{pmatrix} k_1 & 1 \\ 1 & 0 \end{pmatrix} 
	\begin{pmatrix} k_2 & 1 \\ 1 & 0 \end{pmatrix} \dots 
	\begin{pmatrix} k_{2r} & 1 \\ 1 & 0 \end{pmatrix}
	\in \SL_2(\Z), \\
	\abs{W} &:= r, \\
	N(W) &:= \max \{ n \in \Z_{>0} \mid \text{ there exists  an even word } W_1 \text{ such that } W = W_1^n \}.
\end{align}

Moreover, for an even word $ V=(l_1, \dots, l_{2s}) $ and a non-empty even word $ W=(k_1, \dots, k_{2r}) $, define
\begin{align}
	[V \overline{W}] &:= [l_1, \dots, l_{2s}, k_1, \dots, k_{2r}, k_1, \dots, k_{2r}, k_1, \dots, k_{2r}, \dots].
\end{align}

The matrix $ \gamma_{W}^{} $ has the following properties.

\begin{lem}[{\cite[Lemma 2.3]{Mura}}] \label{lem:gamma_W}
	The following properties hold.
	\begin{enumerate}
		\item \label{item:lem:gamma_W1} For even words $ W_1, \dots, W_n $, we have $ \gamma_{W_1 \cdots W_n }^{} = \gamma_{W_1}^{} \dots \gamma_{W_n}^{} $.
		\item \label{item:lem:gamma_W2} Let $ V $, be an even word and $ W \in \Z_{>0}^{\N} $  be an infinite word.
		Then, for an irrational number $ x := [W] $, we have $ \gamma_V^{} x = [V W] $.
		In particular, for a real quadratic number $ v := [\overline{V}] $, we have $ \gamma_V^{} v = v $.
		\item \label{item:lem:gamma_W3} For a reduced real quadratic number $w$, there exists the minimal even word $ W $ such that $ w = [\overline{W}] $.
		Then we have $ N(W)=1 $ and $ \gamma_W^{} = \gamma_{w} $.
		Generally, for a non-empty even word $ W $ and a real quadratic number $ w = [\overline{W}] $,
		we have $ \gamma_W^{} = \gamma_{w}^{N(W)} $.
		\item \label{item:lem:gamma_W4} For a real quadratic number $w$, there exist even words $ V $ and $ W $ such that $ w = [V \overline{W}] $.
		Then $ \gamma_V^{} \gamma_W^{} \gamma_V^{-1} = \gamma_w^{N(W)} $.
	\end{enumerate}
\end{lem}


\subsection{Kaneko's val} \label{subsec:val}


Let $ j \colon \bbH \to \bbC $ be the elliptic modular $ j $-function.
For a real quadratic number $ w $, define
\[
\widetilde{\val}(w) := 
\int_{z_0}^{\gamma_w z_0} j \eta_{w', w}, \quad
\widetilde{1}(w) := 
\int_{z_0}^{\gamma_w z_0} \eta_{w', w}, \quad
\val(w) := \frac{\widetilde{\val}(w)}{\widetilde{1}(w)},
\]
where $ z_0 \in \bbH $ is any point.

The following lemma follows from \cref{lem:eta,lem:eta_on_geod}.

\begin{lem} \label{lem:val_cycle_int}
	For a real quadratic number $ w $, any point $ z_0 \in S_{w', w} $ and any positive integer $ n $, it holds
	\begin{align}
		n \widetilde{\val}(w) 
		&= \int_{z_0}^{\gamma_w^n z_0} j \eta_{w', w}
		= \int_{z_0}^{\gamma_w^n z_0} j ds, \quad
		\\
		n \widetilde{1}(w) 
		&= \int_{z_0}^{\gamma_w^n z_0} \eta_{w', w}
		= \int_{z_0}^{\gamma_w^n z_0} ds
		= d_{\mathrm{hyp}}(z_0, \gamma_w^n z_0).
	\end{align}
\end{lem}

The following lemma is fundamental.

\begin{lem}[{\cite[Propositon 2.7]{Mura}}] \label{lem:tilde_1}
	For a real quadratic number $ w $, it holds
	\[
	\widetilde{1}(w) = 2 \log \veps_w.
	\]
\end{lem}

By the above lemmas, we obtain the following expression of $ \val(w) $, which are suitable to generalize for badly approximable numbers.

\begin{prop} \label{prop:int_rep}
	For a real quadratic number $ w $ and any point $ z_0 \in S_{w', w} $, it holds
	\begin{align}
		\val(w)
		&= \lim_{\substack{ z \in S_{z_0, x} \\ z \to w}}
		\frac{1}{d_{\mathrm{hyp}}(z_0, z)}
		\int_{S_{z_0, z}} j ds.
	\end{align}
\end{prop}

\begin{proof}
	For any positive number $ d \in \R_{>0} $, there exists the unique point $ z_d \in S_{z_0, w} $ such that $ d_{\mathrm{hyp}}(z_0, z_d) = d $.
	Let $ l := d_{\mathrm{hyp}}(z_0, \gamma_w z_0) = \widetilde{1}(w) $.
	It suffices to show that
	\[
	\lim_{n \to \infty} \frac{1}{nl + d} \int_{S_{z_0, z_{nl + d}}} j ds
	= \val(w)
	\]
	for any number $ 0 \le d < l $.
	Since $ \gamma_w^n z_d \in S_{z_0, w} $ and $ d_{\mathrm{hyp}}(z_d, \gamma_w^n z_d) = nl $ by \cref{lem:val_cycle_int}, we have $ z_{nl + d} = \gamma_w^n z_d $.
	Thus, we obtain
	\begin{align}
		\lim_{n \to \infty} \frac{1}{nl + d} \int_{S_{z_0, z_{nl + d}}} j ds
		&=
		\lim_{n \to \infty} \frac{1}{nl + d} \left( \int_{S_{z_0, z_d}} + \int_{S_{z_d, \gamma_w^n z_d}} \right) j ds
		\\
		&=
		\lim_{n \to \infty} \frac{1}{nl + d} \left( \int_{S_{z_0, z_d}} j ds + n \widetilde{\val}(w) \right)
		\\
		&=
		\val(w)
	\end{align}
	by \cref{lem:val_cycle_int}.
\end{proof}

To generalize $ \widetilde{\val}(w) $, $ \widetilde{1}(w) $, and $ \veps_{w} $ for badly approximable numbers, we need some notation.
For a real quadratic number $ w = [V \overline{W}] $, define
\begin{equation} \label{eq:val_hat} 
	\widehat{\val}(w) = \frac{N(W)}{\abs{W}} \widetilde{\val}(w), \quad
	\widehat{1}(w) = \frac{N(W)}{\abs{W}} \widetilde{1}(w), \quad
	\widehat{\veps}_w = \veps_{w}^{N(W)/\abs{W}}.
\end{equation}
These values are independent of a choice of an even word $ W $ and have the following expression.

\begin{prop} \label{prop:word_rep}
	Let $ w $ be a real quadratic number.
	Let $ r \ge 0 $ and $ s \ge 1 $ be integers and $ W_1, \dots, W_{r+s} $ be non-empty even words such that 
	$ w = [W_1 \cdots W_r \overline{W_{r+1} \cdots W_{r+s}}] $. 
	For an integer $ n > r+s $, let $ W_{n} := W_{r+l} $ where $ 0 \le l < s $ is an integer such that $ n \equiv r+l \bmod s $.
	For a positive integer $ i $, let $ \gamma_i := \gamma_{W_i}^{} $. 
	Let $ z_0 \in \bbH $ be a point.
	Then, we have
	\begin{align}
		\widehat{\val}(w)
		&= \lim_{n \to \infty} \frac{1}{n} \int_{z_0}^{\gamma_1 \cdots \gamma_{n} z_0} j \eta_{w', w},
		\\
		\widehat{1}(w)
		&= \lim_{n \to \infty} \frac{1}{n} \int_{z_0}^{\gamma_1 \cdots \gamma_{n} z_0} \eta_{w', w}.
	\end{align}
\end{prop}

\begin{proof}
	Let $ N := N(W_{r+1} \cdots W_{r+s}) $. 
	For a positive integer $ n $, let
	\[
	a_{n} 
	:= \int_{\gamma_1 \cdots \gamma_{n-1} z_0}^{\gamma_1 \cdots \gamma_{n} z_0} j \eta_{w', w}.
	\]
	Since
	\[
	\frac{a_1 + \dots + a_n}{n}
	=
	\frac{1}{n} \int_{z_0}^{\gamma_1 \cdots \gamma_{n} z_0} j \eta_{w', w}, \quad
	\widehat{\val}(w) = \frac{N}{s} \widetilde{\val}(w),
	\]
	it suffices to show that
	\begin{equation} \label{eq:word_rep}
		\lim_{n \to \infty} \frac{a_1 + \dots + a_{r+l+ns}}{r+l+ns}
		= \frac{N}{s} \widetilde{\val}(w)
	\end{equation}
	for any integer $ 0 \le l < s $.
	
	Since $ \gamma_{m+s} = \gamma_m $ for any integer $ m > r+s $, it holds
	\begin{align}
		\gamma_1 \cdots \gamma_{r+l+ns}
		&=
		\left( \gamma_1 \cdots \gamma_{r+l} \gamma_{r+l+1} \cdots \gamma_{r+l+s} (\gamma_1 \cdots \gamma_{r+l})^{-1} \right)^n
		\gamma_1 \cdots \gamma_{r+l}
		=
		\gamma_w^{Nn} \gamma_1 \cdots \gamma_{r+l}
	\end{align}
	by \cref{lem:gamma_W} \cref{item:lem:gamma_W4}.
	Let $ z'_0 := \gamma_1 \cdots \gamma_{r+l} z_0 $.
	Then, we have
	\begin{align}
		a_{r+l} + \dots + a_{r+l+ns}
		&=
		\int_{\gamma_1 \cdots \gamma_{r+l} z_0}^{\gamma_1 \cdots \gamma_{r+l+ns} z_0}
		j \eta_{w', w}
		\\
		&=
		\int_{z'_0}^{\gamma_w^{Nn} z'_0}
		j \eta_{w', w}
		\\
		&=
		Nn \widetilde{\val}(w)
	\end{align}
	by \cref{lem:val_cycle_int}.
	Thus, we obtain \cref{eq:word_rep}.
	The second equality is similarly proved.
\end{proof}

\begin{rem}
	In \cref{prop:word_rep}, the first finite term of a continued fraction does not contribute to the left-hand side.
	Thus, the right-hand side in \cref{prop:word_rep} depends only on the period in the continued fraction.
\end{rem}


\section{Fundamental properties of the extended $ \val $ function} \label{sec:fund}



With reference to \cref{prop:int_rep,prop:word_rep}, we now define $ \val(x) $ for a badly approximable number $ x $ as follows.
Let $ W_1, W_2, \dots $ be an infinite sequence of even words such that $ x= [W_1 W_2 \cdots] $ and the set $ \{ W_1, W_2, \dots  \} $ is finite.
We remark that if $ x < 1 $ then the first entry of $ W_1 $ is not positive.
Let $ c_n \in \Z_{>0} $ be the denominator of the rational number $ [W_1 \cdots W_n] $ and put
\[
L := \lim_{n \to \infty} \frac{\abs{W_1} + \dots + \abs{W_n}}{n}.
\]
Take $ x' \in \bbP^1(\R) \setminus \{ x \} $ and $ z_0 \in \bbH $.
We define the limits $ \val(x), \widehat{\val}(x), \widehat{1}(x), $ and $ \widehat{\veps}_x $ by
\begin{align}
	\val(x) &:= \lim_{
		\substack{ z \in S_{z_0, x} \\
			z \to x}}
	\frac{1}{d_{\mathrm{hyp}}(z_0, z)}
	\int_{S_{z_0, z}} j ds, 
	\label{eq:val_new1} \\
	\widehat{\val}(x) &:=
	\lim_{n \to \infty} \frac{1}{\abs{W_1} + \dots + \abs{W_n}}
	\int_{z_0}^{\gamma_{W_1 \cdots W_n}^{} z_0} j \eta_{x', x},
	\label{eq:val_new2}\\
	\widehat{1}(x) &:=
	\lim_{n \to \infty} \frac{1}{\abs{W_1} + \dots + \abs{W_n}}
	\int_{z_0}^{\gamma_{W_1 \cdots W_n}^{} z_0} \eta_{x', x},
	\label{eq:val_new3} \\
	\widehat{\veps}_x &:= \lim_{n \to \infty} c_n^{1/Ln}.
	\label{eq:val_new4}
\end{align}

\begin{figure}[hbtp]
	\centering
	\includegraphics[clip,width = 15.0cm]{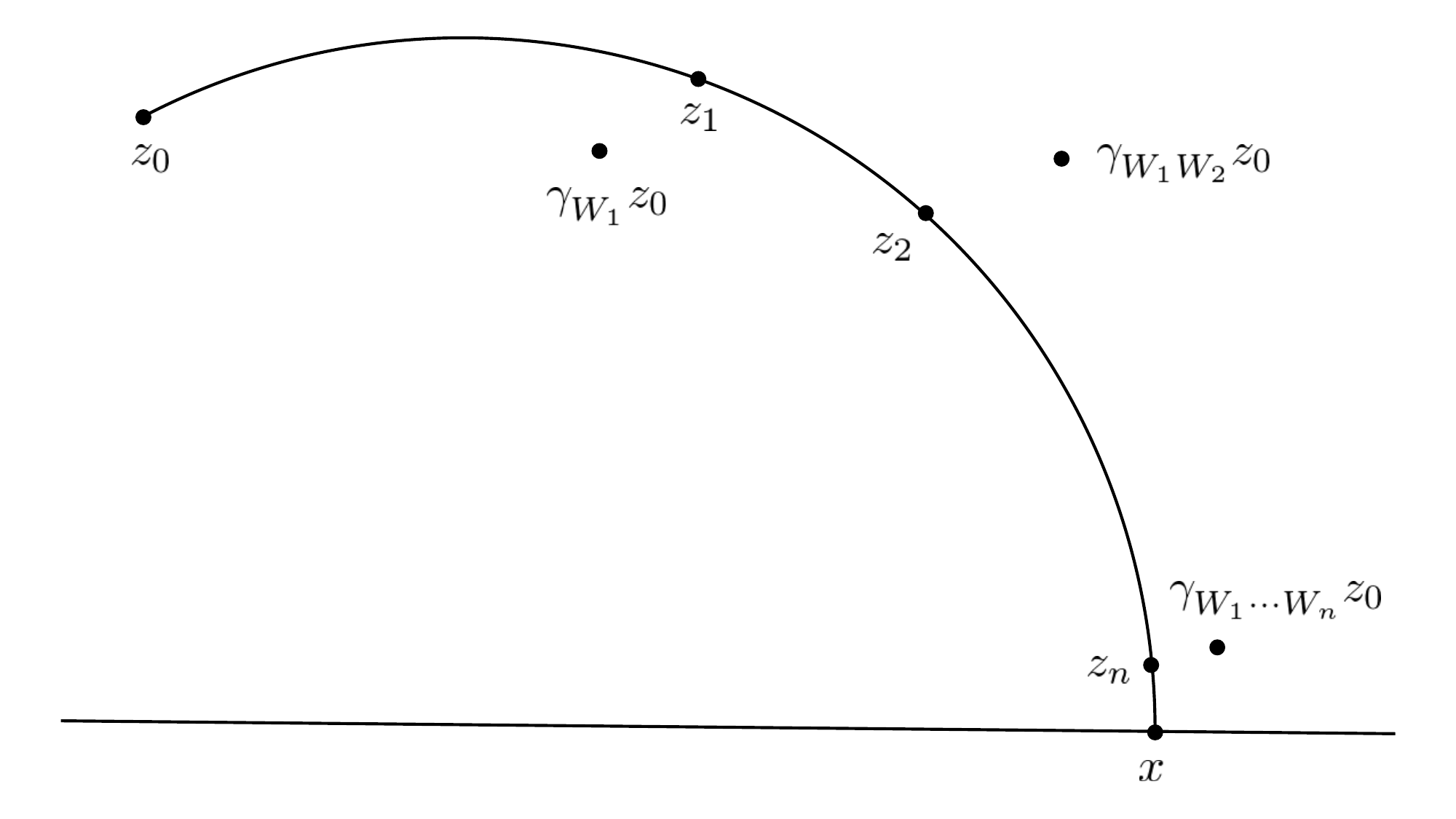}
	\caption{The paths for the definitions of $ \val(x), \widehat{\val}(x) $, and $ \widehat{1}(x) $ in \cref{eq:val_new1}, \cref{eq:val_new2}, and \cref{eq:val_new3}}
	\label{fig:main}
\end{figure}

In the later sections, we will prove the following.

\begin{thm} \label{thm:main:extension}
	Let $ x $ be a badly approximable number and $ W_1, W_2, \dots $ be an infinite sequence of even words such that $ x= [W_1 W_2 \cdots] $ and the set $ \{ W_1, W_2, \dots  \} $ is finite.
	Take any points $ x' \in \bbP^1(\R) \setminus \{ x \} $ and $ z_0 \in \bbH $.
	Then the following statements hold.
	\begin{enumerate}
		\item \label{item:thm:main1}
		The limit values $ \val(x), \widehat{\val}(x), \widehat{1}(x), $ and $ \widehat{\veps}_x $ defined in each of \cref{eq:val_new1}, \cref{eq:val_new2}, \cref{eq:val_new3}, and \cref{eq:val_new4} are bounded.
		Moreover, if they converge then they are independent of $ x', z_0 $ and $ W_1, W_2, \dots $.
		Further, they are $ \SL_2(\Z) $-invariant.
		\item \label{item:thm:main2}
		If $ x $ is a real quadratic number, then $ \val(x), \widehat{\val}(x), \widehat{1}(x), $ and $ \widehat{\veps}_x $ converge and coincide with the values defined in \cref{eq:def_val} and \cref{eq:val_hat}.
		\item \label{item:thm:main3}
		If $ \widehat{\val}(x) $ and $ \widehat{1}(x) $ converge, then
		for any point $ x' \in \bbP^1(\R) \setminus \{ x \} $ and a sequence $ \{ x_n \}_{n=1}^\infty $ of real numbers with
		$ z_n := \gamma_{W_1 \cdots W_n}^{} (x_n + \iu) \in S_{\infty, x}, z_n \to x $, we have
		\begin{align}
			\widehat{\val}(x) &= \lim_{n \to \infty} \frac{1}{\abs{W_1} + \dots + \abs{W_n}} \int_{z_0}^{z_n} j \eta_{x', x}, \\
			\widehat{1}(x) &= \lim_{n \to \infty} \frac{1}{\abs{W_1} + \dots + \abs{W_n}} \int_{z_0}^{z_n} \eta_{x', x}.
		\end{align}
		Here we define $ 1/(z - \infty) := 0 $ if $ x' = \infty $.
		\item \label{item:thm:main4}
		The limit $ \widehat{\veps}_x $ converges if and only if $ \widehat{1}(x) $ converges.
		Moreover, we have $ \widehat{1}(x) = 2 \log \widehat{\veps}_x $.
		\item \label{item:thm:main5}	
		It holds $ \widehat{\veps}_x \ge (3 + \sqrt{5})/2 $.
		In particular, $ \widehat{1}(x) > 0 $.
		\item \label{item:thm:main6}
		If $ \widehat{\val}(x) $ and $ \widehat{1}(x) $ converge, then $ \val(x) $ converges and
		we have $ \val(x) = \widehat{\val}(x)/\widehat{1}(x) $.
	\end{enumerate}
\end{thm}

Here we remark that \cref{thm:main:extension} \ref{item:thm:main6} follows from \cref{thm:main:extension} \cref{item:thm:main4} and \cref{item:thm:main5}.

\begin{figure}[bthp]
	\centering
	\includegraphics[clip,width = 15.0cm]{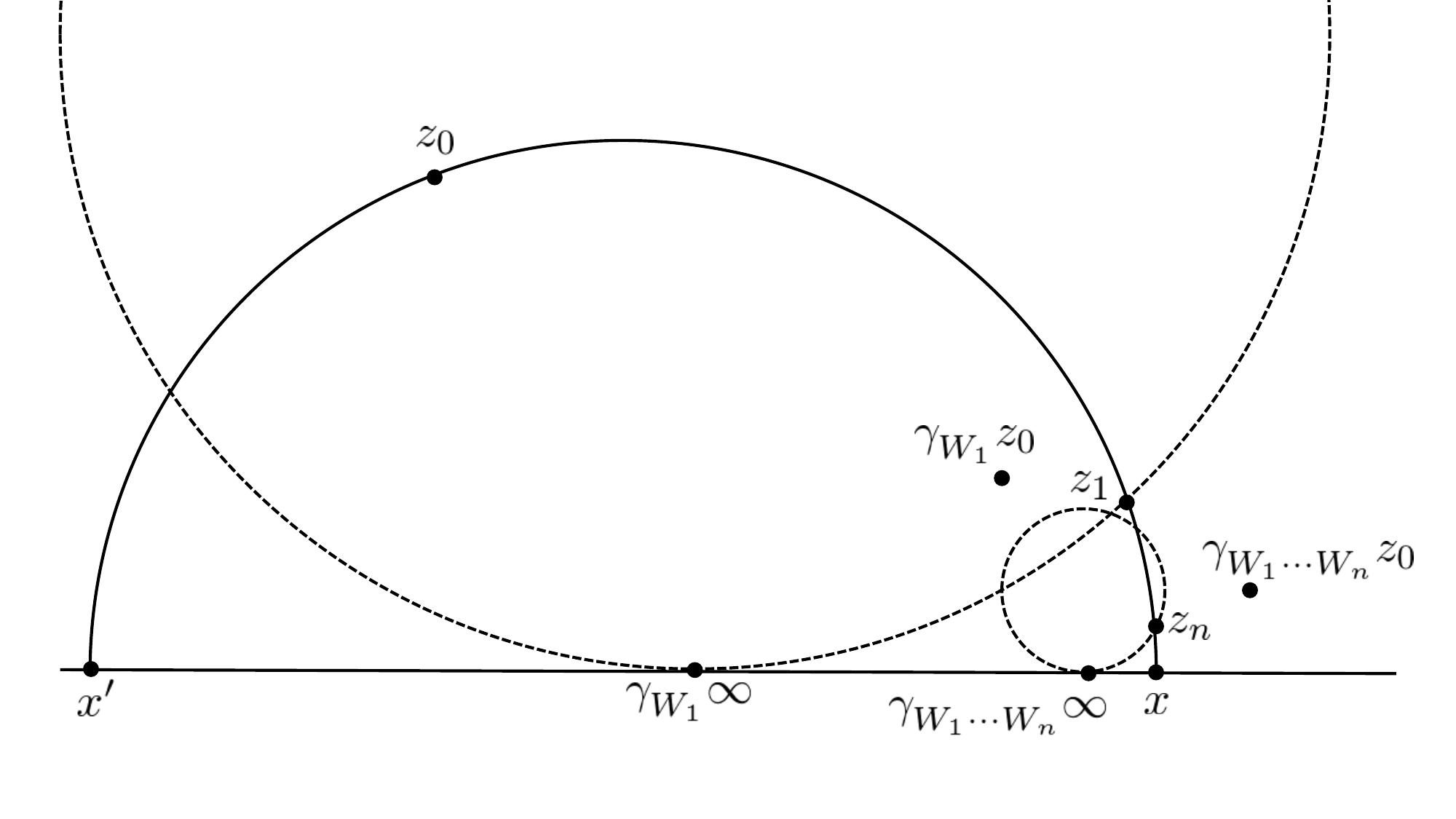}
	\caption{The path of integrals for the limit values in \cref{thm:main:extension} \cref{item:thm:main3}}
	\label{fig:coincidence}
\end{figure}


\section{The extended $ \val $ function as the limit value along a geodesic} \label{sec:def_geodesic}


In this section, we prove \cref{thm:main:extension} \ref{item:thm:main1} and \ref{item:thm:main2} for $ \val(x) $.
We start with the following properties of badly approximable numbers.
Let $ \pi \colon \bbH \to \SL_2(\Z) \backslash \bbH $ be the natural projection.

\begin{prop}[{\cite[Chapter VII, Theorem 3.4]{Dal}}, {\cite[Chapter VII, Lemma 3.5]{Dal}}] \label{prop:bad_appr}
	For an irrational number $ x \in \R \setminus \Q $, the following statements are equivalent. 
	\begin{enumerate}
		\item $ x $ is badly approximable;
		\item $ L(x)  := \sup \left\{ L > 0 \relmiddle| \exists p_n, q_n \in \Z \text{ s.t. } (p_n, q_n) = 1, q_n \to \infty, 
		\abs{ x - p_n/q_n } \le 1/L q_n^2 \right\}
		< \infty $;
		\item $ h(x) := \sup \left\{ t > 0 \mid \exists z_n \in S_{z_0, x} \text{ s.t. } z_n \to x, 
		\pi(z_n) \in \pi(\{ z \in \bbH \mid \ImNew(z) = t \}) \right\}
		< \infty $;
		\item For any point $ z_0 \in \bbH $, the closure $ \overline{\pi(S_{z_0, x})} \subset \SL_2(\Z) \backslash \bbH $ is compact. 
	\end{enumerate}
\end{prop}

The constant $ L(x) $ is called the Lagrange number of $ x $. 
Usually, badly approximable numbers are defined as irrational numbers with finite Lagrange numbers.

The following is a key fact in this section.

\begin{thm}[{\cref{thm:main:extension} \cref{item:thm:main1} and \cref{item:thm:main2} for $ \val(x) $}] \label{thm:new_path}
	Let $ x $ be a badly approximable number.
	Take any point $ z_0 \in \bbH $.
	Then the following statements hold.
	\begin{enumerate}
		\item \label{item:thm:new_path1} 
		The limit value
		\[
		\val(x) := \lim_{
			\substack{ z \in S_{z_0, x} \\
				z \to x}}
		\frac{1}{d_{\mathrm{hyp}}(z_0, z)}
		\int_{S_{z_0, z}} j ds
		\]	
		is bounded.
		\item \label{item:thm:new_path2} 
		For any matrix $ \gamma \in \SL_2(\Z) $, we have $ \val(\gamma x) = \val(x) $.		
		\item \label{item:thm:new_path3} 
		If the limit $ \val(x) $ converges, then it is independent of $ z_0 $.
		\item \label{item:thm:new_path4} 
		For a real quadratic number $ x $, $ \val(x) $ converges and coincides with the value defined in \cref{eq:def_val}.
	\end{enumerate}
\end{thm}

\begin{proof}
	To begin with, we remark that \cref{item:thm:new_path4} follows from \cref{prop:int_rep,item:thm:new_path3}.
	
	We will prove \ref{item:thm:new_path1}. 	
	Since $ x $ is badly approximable, the closure $ \overline{\pi(S_{x, z_0})} \subset \SL_2(\Z) \backslash \bbH $ is compact by \cref{prop:bad_appr}.
	Thus,  we can choose $ K>0 $ such that $ \abs{j(z)} \le K $ on $ S_{z_0, x} $.
	Since
	\begin{align}
		\abs{\frac{1}{d_{\mathrm{hyp}}(z_0, z)} 	\int_{S_{z_0, z}} j ds} 
		= \frac{1}{d_{\mathrm{hyp}}(z_0, z)} 	\int_{S_{z_0, z}} K ds
		= K,
	\end{align}
	the limit is bounded.
	
	The second claim \ref{item:thm:new_path2} follows from the fact that $ j(z) $ is $ \SL_2(\Z) $-invariant.
	
	As for the third claim \ref{item:thm:new_path3}, for each point $ z_0 \in \bbH $, we consider
	\[
	\val(x, z_0) := \lim_{
		\substack{ z \in S_{z_0, x} \\
			z \to x}}
	\frac{1}{d_{\mathrm{hyp}}(z_0, z)}
	\int_{S_{z_0, z}} j ds.
	\]
	We will show that this value is independent of $ z_0 $ in two steps.
	
	\textbf{Step 1}.
	We will show that $ \val(x, z_0) = \val(x, z_1) $ for any point $ z_1 \in S_{z_0, x} $.
	For $ i \in \{0, 1\} $, let
	\[
	a_i(z) := \int_{S_{z_i, z}} j ds, \quad b_i(z) := \int_{S_{z_i, z}} ds.
	\]
	Clearly,
	\[	
	\lim_{
		\substack{ z \in S_{z_0, x} \\
			z \to x}}
	b_i(z)
	= \infty.
	\]
	Since
	\[
	a_0(z) - a_1(z) = \int_{S_{z_0, z_1}} j ds, \quad b_0(z) - b_1(z) = \int_{S_{z_0, z_1}} ds
	\]
	is bounded, we have
	\[
	\val(x, z_0)
	= \lim_{
		\substack{ z \in S_{z_0, x} \\
			z \to x}}
	\frac{a_0(z)}{b_0(z)}
	= \lim_{
		\substack{ z \in S_{z_0, x} \\
			z \to x}}
	\frac{a_1(z) + O(1)}{b_1(z) + O(1)}
	= \lim_{
		\substack{ z \in S_{z_0, x} \\
			z \to x}}
	\frac{a_1(z)}{b_1(z)}
	= \val(x, z_1).
	\]
	
	\begin{figure}[hbtp]
		\centering
		\includegraphics[clip,width = 15.0cm]{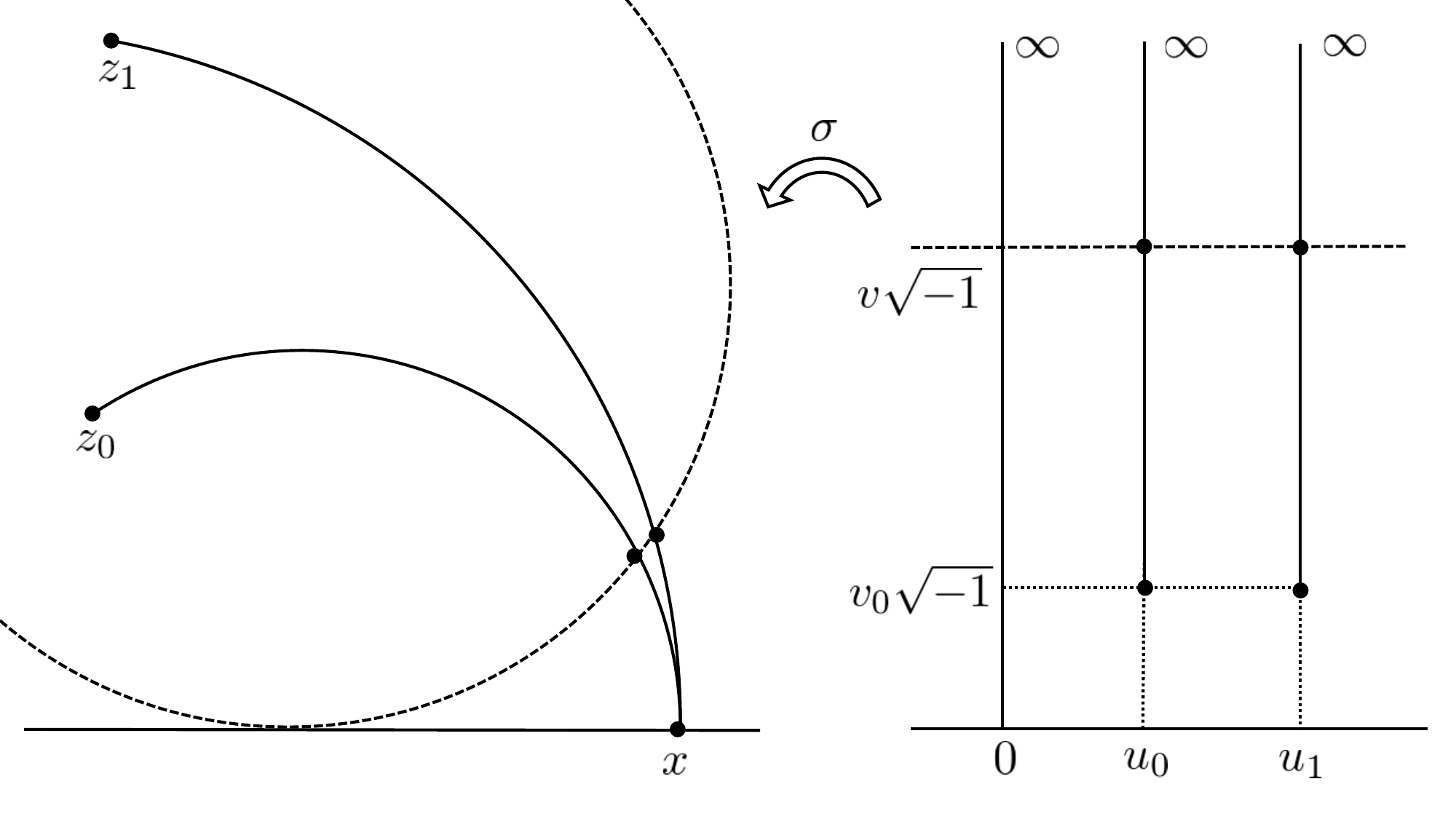}
		\caption{The paths of integrals appearing in Step 2 in the proof of \cref{thm:new_path} \ref{item:thm:new_path3}}
		\label{fig:new_path}
	\end{figure}
	
	\textbf{Step 2}.
	We will show that $ \val(x, z_0) = \val(x, z_1) $ for any points $ z_0, z_1 \in \bbH $.
	This part is due to Matsusaka.
	Take a matrix $ \sigma \in \SL_2(\R) $ such that $ \sigma^{-1} x = \infty $.
	Let $ v_0 $ be any positive number.
	By Step 1, we can replace $ z_i $ by $ \Re(z_i) + v_0 \iu $.
	Thus,  we may assume $ v_0 = \Im(\sigma^{-1} z_0) = \Im(\sigma^{-1} z_1) $. 	
	Let $ u_i := \ReNew(\sigma^{-1} z_i) $ for $ i \in  \{ 0, 1 \} $.
	Then $ \sigma^{-1} S_{z_i, x} $ is a subgeodesic of $ S_{u_i , \infty} $.
	Since $ ds = y^{-1} \sqrt{dx^2 + dy^2} = y^{-1}dy $ on $ S_{u_i, \infty} $, we have
	\begin{align}
		\val(x, z_i)
		&= \lim_{
			\substack{ z \in S_{z_i, x} \\
				z \to x}}
		\frac{1}{d_{\mathrm{hyp}}(z_i, z)}
		\int_{S_{z_i, z}} j ds\\
		&= \lim_{v \to \infty}
		\frac{1}{d_{\mathrm{hyp}}(\sigma(u_i + v\iu), \sigma(u_i + v_0 \iu))}
		\int_{v_0}^{v} j(\sigma(u_i + y\iu)) \frac{dy}{y}.
	\end{align}
	Since
	\begin{align}
		d_{\mathrm{hyp}}(\sigma(u_i + v\iu), \sigma(u_i + v_0 \iu))
		= d_{\mathrm{hyp}}(v\iu, v_0 \iu) 
		= \int_{v_0}^{v} \frac{dy}{y} 
		= \log \frac{v}{v_0},
	\end{align}
	we obtain	
	\[
	\val(x, z_0) - \val(x, z_1)
	= \lim_{v \to \infty}
	\frac{1}{\log v/v_0}
	\int_{v_0}^{v} \left( j(\sigma(u_0 + y\iu)) - j(\sigma(u_1 + y\iu)) \right) \frac{dy}{y}.
	\]
	For two points $ z, z' \in \bbH $, the hyperbolic distances on $ \bbH $ and $ \SL_2(\Z) \backslash \bbH $ are respectively defined by
	\[
	d_{\mathrm{hyp}}(z, z') := \length(S_{z, z'}), \quad
	d_{\SL_2(\Z) \backslash \bbH}(\pi(z), \pi(z')) 
	:= \min \left\{ d_{\mathrm{hyp}}(\gamma z, z') \mid \gamma \in \SL_2(\Z) \right\}.
	\]
	They define the hyperbolic distances on $ \bbH $ and $ \SL_2(\Z)\backslash\bbH $ respectively.
	By \cref{lem:distance}, we have
	\[
	\cosh d_{\mathrm{hyp}}(\sigma(u_0 + y\iu), \sigma(u_1 + y\iu))
	= \cosh d_{\mathrm{hyp}}(u_0 + y\iu, u_1 + y\iu) 
	= 1 + \frac{\abs{u_0 - u_1}^2}{2y^2}.
	\]
	Thus, we have
	\[
	\lim_{y \to \infty} d_{\mathrm{hyp}}(\sigma(u_0 + y\iu), \sigma(u_1 + y\iu)) = 0.
	\]
	Since $ x $ is badly approximable,
	$ \overline{\pi(S_{x, z_0})} \cup \overline{\pi(S_{x, z_1})} 
	\subset \SL_2(\Z) \backslash \bbH $
	is compact by \cref{prop:bad_appr}.
	Thus, $ j(z) $ is continuous on $ \overline{\pi(S_{x, z_0})} \cup \overline{\pi(S_{x, z_1})} $ with respect to the metric $ d_{\SL_2(\Z) \backslash \bbH} $.
	Hence for any $ \veps > 0 $, there exists $ v_0 > 0 $ such that
	\[
	\abs{ j(\sigma(u_0 + y\iu)) - j(\sigma(u_1 + y\iu)) } < \veps
	\]
	for any $ y > v_0 $.
	Then we have
	\[
	\abs{\val(x, z_0) - \val(x, z_1)}
	\le \lim_{v \to \infty}
	\frac{1}{\log v/v_0}
	\int_{v_0}^{v} \veps  \frac{dy}{y}
	= \veps,
	\]
	that is, $ \val(x, z_0) = \val(x, z_1) $. 
\end{proof}


\section{The extended $ \val $ function as the limit value along a continued fraction} \label{sec:def_word}


In this section, we prove \cref{thm:main:extension} \ref{item:thm:main1}, \ref{item:thm:main2}, and \ref{item:thm:main3} for $ \widehat{\val}(x) $ and $ \widehat{1}(x) $.

To prove \cref{thm:main:extension} \ref{item:thm:main2}, we prepare several lemmas.

\begin{lem} \label{lem:C_w}
	For a badly approximable number $ x = [k_1, k_2, \dots] $, let 
	$ \gamma_n := \gamma_{(k_1, \dots, k_{2n})} $.
	Then for any point $ x' \in \bbP^1(\R) $, the set 
	$ \{ \gamma_n^{-1} x' \mid n \in \Z_{>0} \}  $ is bounded. 
\end{lem}

\begin{proof}
	Let $ h(x) $ be the number defined in \cref{prop:bad_appr}.
	By \cref{prop:bad_appr}, we have
	\[
	\gamma_n^{-1} S_{x', x} \cap \{ z \in \bbH \mid \ImNew(z) \ge h(x) + 1 \}  = \emptyset 
	\]
	for all sufficiently large $ n $.
	Thus,  the radius of $ \gamma_n^{-1} S_{x', x} $ is less than $ h(x) + 1 $.
	Since coefficients of the continued fraction expansion of $ x $ are bounded, 
	$ \{ \gamma_n^{-1} x \mid n \in \Z_{>0} \} $ is bounded. 
	Thus,  $ \{ \gamma_n^{-1} x' \mid n \in \Z_{>0} \} $ is also bounded.
\end{proof}

\begin{figure}[hbtp]
	\centering
	\includegraphics[clip,width = 15.0cm]{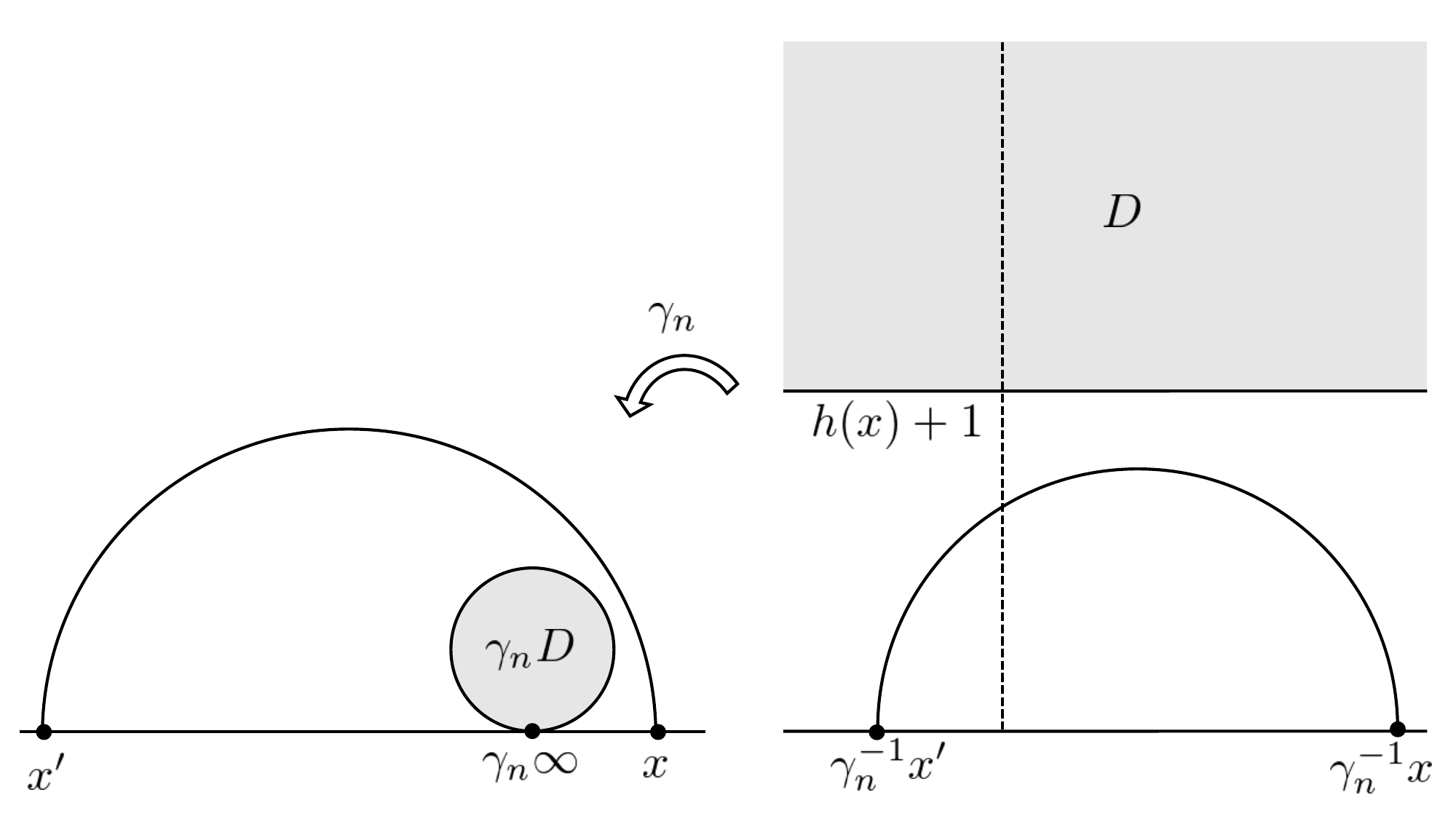}
	\caption{Geodesics in the proof of \cref{lem:C_w}}
	\label{fig:C_w}
\end{figure}

\begin{lem} \label{lem:triv_eval_integral}
	For an even word $ W $, a point $ z_0 \in \bbH $, and sequences $ \{x_n\}, \{x'_n\}, \{y_n\}, \{y'_n\} $, we have
	\[
	\int_{z_0}^{\gamma_W^{} z_0} j \left( \eta_{x_n', x_n} - \eta_{y_n', y_n} \right)
	= O(\abs{x_n - y_n} + \abs{x'_n - y'_n}).
	\]
\end{lem}

\begin{proof}
	Since 
	\[
	\abs{\dfrac{1}{z - x_n} - \dfrac{1}{z - y_n}}
	\le \frac{\abs{x_n - y_n}}{\min \{\ImNew(z_0), \ImNew(\gamma_W^{} z_0)\}}
	\]
	on the contour $ \{ t z_0 + (1-t) \gamma_W^{} z_0 \mid 0 \le t \le 1 \} $, it holds
	\[
	\abs{\int_{z_0}^{\gamma_W^{} z_0} j(z) \left( \dfrac{1}{z - x_n} - \dfrac{1}{z - y_n} \right)}
	\le \frac{\abs{x_n - y_n}}{\min \{\ImNew(z_0), \ImNew(\gamma_W^{} z_0)\}}
	\abs{\int_{z_0}^{\gamma_W^{} z_0} j(z) dz}
	= O(\abs{x_n - y_n}).
	\]
	Similar argument shows
	\[
	\abs{\int_{z_0}^{\gamma_W^{} z_0} j(z) \left( \dfrac{1}{z - x_n'} - \dfrac{1}{z - y_n'} \right)}
	= O(\abs{x_n' - y_n'}).
	\]
\end{proof}

A key result in this section is the following.

\begin{thm}[\cref{thm:main:extension} \ref{item:thm:main1} for $ \widehat{\val}(x) $ and $ \widehat{1}(x) $] \label{thm:new_word}
	Let $ x $ be a badly approximable number and $ W_1, W_2, \dots $ be an infinite sequence of even words such that $ x= [W_1 W_2 \cdots] $ and the set $ \{ W_1, W_2, \dots  \} $ is finite.
	Take a point $ z_0 \in \bbH $.
	Then the following statements hold.
	\begin{enumerate}
		\item \label{item:thm:new_word1}
		The sequences defining the limits
		\begin{align}
			\widehat{\val}(x) &:=
			\lim_{n \to \infty} \frac{1}{\abs{W_1} + \dots + \abs{W_n}}
			\int_{z_0}^{\gamma_{W_1 \cdots W_n}^{} z_0} j \eta_{\infty, x}, \\
			\widehat{1}(x) &:=
			\lim_{n \to \infty} \frac{1}{\abs{W_1} + \dots + \abs{W_n}}
			\int_{z_0}^{\gamma_{W_1 \cdots W_n}^{} z_0} \eta_{\infty, x}
		\end{align}
		is bounded.
		\item \label{item:thm:new_word2}
		The limits $ \widehat{\val}(x) $ and $ \widehat{1}(x) $ are independent of $ z_0 $.		
		\item \label{item:thm:new_word3}
		The limits $ \widehat{\val}(x) $ and $ \widehat{1}(x) $ are independent of $ \{ W_n \mid n \ge 1 \} $.		
		\item \label{item:thm:new_word4}
		The limits $ \widehat{\val}(x) $ and $ \widehat{1}(x) $ are $ \SL_2(\Z) $-invariant.
	\end{enumerate}
\end{thm}

\begin{proof}
	For the first claim \ref{item:thm:new_word1}, let
	\begin{align}
		L &:= \lim_{n \to \infty} \frac{\abs{W_1} + \dots + \abs{W_n}}{2n}, \\
		a_{n}(z_0) &:= \int_{\gamma_{W_{1} \cdots W_{n-1}^{}z_0}}^{\gamma_{W_{1} \cdots W_n}^{} z_0} j \eta_{\infty, x}, \\
		K &:= \max \left\{ \abs{j(z)} \relmiddle| \abs{\ReNew(z)} \le 1/2, \abs{z} \le 1, \Im(z) \le h(x) \right\}, \\
		R &:= \min \{ \ImNew(z_0), \ImNew(\gamma_{W_{n}}^{} z_0) \mid n \in \Z_{>0} \}, \\
		M &:= \sup \{ \gamma_{W_1 \cdots W_n}^{-1}x, \gamma_{W_1 \cdots W_n}^{-1}\infty \mid n \in \Z_{>0} \}.
	\end{align}
	Here $ K $ and $ R $ are positive real numbers and $ L $ converges since the set $ \{ W_1, W_2, \dots \} $ is finite.
	Since $ x $ is badly approximable, $ M $ is a positive real number by \cref{lem:C_w}. 
	We can write
	\[
	\widehat{\val}(x) 
	= \frac{1}{L} \lim_{n \to \infty} \frac{a_{1}(z_0) + \dots + a_{n}(z_0)}{n}.
	\]
	For any positive integer $ n_0 $, we have
	\begin{align}
		\abs{ a_{1}(z_0) + \dots + a_{n}(z_0) }
		&\le \int_{z_0}^{\gamma_{W_{1} \cdots W_{n}}^{} z_0} \abs{j \eta_{\infty, x}} \\
		&= \int_{z_0}^{\gamma_{W_{1} \cdots W_{n}}^{} z_0} \abs{j(z) \frac{-1}{z-x} } \abs{dz} \\
		&\le \frac{2MK}{R^2} n
	\end{align}
	and thus the sequence $ \{ (a_{1}(z_0) + \dots + a_{n}(z_0))/n \}_{n \ge 1} $ is bounded.
	
	For \cref{item:thm:new_word2}, pick any other point $ z_0' \in \bbH $.
	We have
	\begin{align}
		\left( a_{1}(z_0) + \dots + a_{n}(z_0) \right) - \left( a_{1}(z'_0) + \dots + a_{n}(z'_0) \right) 
		&= \left( \int_{\gamma_{W_{1} \cdots W_n}^{} z_0'}^{\gamma_{W_{1} \cdots W_n}^{} z_0}
		- \int_{z_0'}^{z_0} \right) j \eta_{\infty, x}
		\\
		&= \int_{z_0'}^{z_0} j \left( \gamma_{W_{1} \cdots W_n}^{*} \eta_{\infty, x} - \eta_{\infty, x} \right)
		\\
		&= \int_{z_0'}^{z_0} j \left( \eta_{\gamma_{W_{1} \cdots W_n}^{-1} \infty, \gamma_{W_{1} \cdots W_n}^{-1} x} - \eta_{\infty, x} \right)
	\end{align}
	by \cref{lem:eta}.
	Since the most right-hand side is bounded by \cref{lem:C_w,lem:triv_eval_integral}, we have
	\[
	\lim_{n \to \infty} \frac{a_{1}(z_0) + \dots + a_{n}(z_0)}{n} 
	= \lim_{n \to \infty} \frac{a_{1}(z'_0) + \dots + a_{n}(z'_0)}{n}.
	\]
	
	As for \ref{item:thm:new_word3}, let $ x = [k_1, k_2, \dots], V_n := (k_{2n-1}, k_{2n}) $, and $ L_n := \abs{W_1} + \dots + \abs{W_n} $.
	Since $ \gamma_{W_{1} \cdots W_n}^{}  = \gamma_{V_{1} \cdots V_{L_n}}^{} $, we have
	\[
	\frac{1}{L_n} \int_{z_0}^{\gamma_{W_{1} \cdots W_n}^{} z_0}
	= \frac{1}{L_n} \int_{z_0}^{\gamma_{V_{1} \cdots V_{L_{n}}}^{} z_0}
	\]
	and thus the values $ \widehat{\val}(x) $ and $ \widehat{1}(x) $ are independent of $ \{ W_n \mid n \ge 1 \} $.	
	
	Finally, we prove \ref{item:thm:new_word4}.	
	Since $ \SL_2(\Z) $-equivalent real numbers have the same continued fraction expansions except for the first few terms, it suffices to show 
	$ \widetilde{f}(x) = \widetilde{f}(\gamma^{-1} x) $
	in the case when $ x= [W_1 W_2 \dots], \gamma = \gamma_{W_{1}}^{} \dots \gamma_{W_{k}}^{}  $.
	This follows from the fact that
	\[
	\int_{\gamma_{W_{1}}^{} \dots \gamma_{W_{m}}^{} z_0}^{\gamma_{W_{1}}^{} \dots \gamma_{W_{m+n+k}}^{} z_0} j \eta_{\infty, x}
	= \int_{\gamma_{W_{k+1}}^{} \dots \gamma_{W_{m}}^{} z_0}^{\gamma_{W_{k+1}}^{} \dots \gamma_{W_{m+n+k}}^{} z_0}
	j \eta_{\gamma^{-1} \infty, \gamma^{-1} x}
	\]
	for $ m>k $ is bounded by \cref{lem:C_w,lem:triv_eval_integral}.
\end{proof}

In \cref{thm:new_word} \ref{item:thm:new_word1}, we consider the differential form $ \eta_{\infty, x} $.
In fact, we can replace it by  $ \eta_{x', x} $ for any point $ x' \in \bbP^1(\R) \setminus \{ x \} $.
To prove it, we prepare the following lemma.

\begin{lem} \label{lem:x'_indep}
	In the setting of Theorem $ \ref{thm:new_word} $, let $ \{x'_n\}_{n=1}^\infty \subset \R $ be a bounded sequence such that $ x $ is not an accumulation point.
	Then the sequence
	\[
	\int_{z_0}^{\gamma_{W_1 \cdots W_n}^{} z_0} j(z) \frac{dz}{z - x'_n}
	\]
	is bounded.
\end{lem}

\begin{proof}
	Let
	\[
	K := \max \left\{ \abs{j(z)} \relmiddle| \abs{\ReNew(z)} \le 1/2, \abs{z} \le 1, \Im(z) \le h(x) \right\}.
	\]
	Then
	\begin{align}
		\abs{\int_{z_0}^{\gamma_{W_1 \cdots W_n}^{} z_0} j(z) \frac{dz}{z - x'_n}}
		\le K \int_{S_{z_0, \gamma_{W_{1} \cdots W_{n}}^{} z_0}} \frac{\abs{dz}}{\abs{z - x'_n}} \\
		\le K \int_{S_{z_0, x}} \frac{\abs{dz}}{\abs{z - x'_n}}.
	\end{align}
\end{proof}

By \cref{lem:x'_indep}, we can reformulate the definition of $ \widehat{\val}(x) $ in \cref{thm:new_word} \ref{item:thm:new_word1} as follows.

\begin{prop} \label{prop:x'_indep}
	In the setting of Theorem $ \ref{thm:new_word} $, let $ \{x'_n\}_{n=1}^\infty \subset \R $ be a bounded sequence such that $ x $ is not an accumulation point.
	Then 
	\[
	\widehat{\val}(x) =
	\lim_{n \to \infty} \frac{1}{\abs{W_1} + \dots + \abs{W_n}}
	\int_{z_0}^{\gamma_{W_1 \cdots W_n}^{} z_0} j \eta_{x'_n, x}^{}
	\]
	holds.
\end{prop}

Finally, we prove \cref{thm:main:extension} \ref{item:thm:main3}.

\begin{thm}[\cref{thm:main:extension} \ref{item:thm:main3}] \label{pro:coincidence}
	In the setting of Theorem $ \ref{thm:new_word} $, for any point $ x' \in \bbP^1(\R) \setminus \{ x \} $ and a sequence $ \{ x_n \}_{n=1}^\infty $ of real numbers such that $ z_n := \gamma_{W_1 \cdots W_n}^{} (x_n + \iu) \in S_{\infty, x} $ and $ z_n \to x $, we have
	\begin{align}
		\widehat{\val}(x) &= \lim_{n \to \infty} \frac{1}{\abs{W_1} + \dots + \abs{W_n}} \int_{z_0}^{z_n}  j \eta_{x', x}, \\
		\widehat{1}(x) &= \lim_{n \to \infty} \frac{1}{\abs{W_1} + \dots + \abs{W_n}} \int_{z_0}^{z_n} \eta_{x', x}.
	\end{align}
\end{thm}

\begin{proof}
	We prove only the first equality.
	Let $ \gamma_n := \gamma_{W_1 \cdots W_n}^{} $. 
	By \cref{prop:x'_indep}, we have
	\[
	\widehat{\val}(x) = \lim_{n \to \infty} \frac{1}{\abs{W_1} + \dots + \abs{W_n}} 
	\int_{\iu}^{\gamma_n \iu} j \eta_{x', x}.
	\]
	Take a point $ z_0 \in S_{x', x} $ such that $ \pi(z_0) \in \pi(\{ z \in \bbH \mid 1 \le \ImNew(z) \le h(x) \}) $.
	Then we have
	\[
	\widehat{\val}(x) = \lim_{n \to \infty} \frac{1}{\abs{W_1} + \dots + \abs{W_n}} 
	\int_{z_0}^{\gamma_n \iu} j \eta_{x', x}.
	\]
	Since $ x_n + \iu \in \gamma_n^{-1}(S_{x', x}) $, we have
	$ \abs{x_n} \le \max \{ \abs{\gamma_n^{-1} x}, \abs{\gamma_n^{-1} x'} \} $.
	Thus,  $ \abs{x_n} $ is bounded by \cref{lem:C_w}. 
	Since
	\[
	\abs{ d(z_n, z_0) - d(\gamma_n \iu, z_0) }
	\le d(z_n, \gamma_n \iu)
	\le d(x_n + \iu, \iu) 
	= 1 + \frac{x_n^2}{2}
	\]
	by the triangle inequality and \cref{lem:distance} \cref{item:lem:distance2},
	$ d(z_n, z_0) - d(\gamma_n \iu, z_0) $ is bounded. 
	Let
	\[
	K := \max \left\{ \abs{j(z)} \relmiddle| \abs{\ReNew(z)} \le 1/2, \abs{z} \le 1, \Im(z) \le h(x) \right\}.
	\]
	Then
	\begin{align}
		\left( \int_{z_0}^{z_n} - \int_{z_0}^{\gamma_n \iu} \right) j \eta_{x', x} 
		&= \int_{\iu}^{x_n + \iu} j \eta_{\gamma_n^{-1}x', \gamma_n^{-1}x} \\
		&\le K \int_{\iu}^{x_n + \iu} \frac{\abs{\gamma_n^{-1}x - \gamma_n^{-1}x'}}{(\min\{\ImNew(x_n + \iu), \ImNew(\iu)\})^2} dz \\
		&\le K \abs{\gamma_n^{-1}x - \gamma_n^{-1}x'} x_n 
	\end{align}
	is bounded by \cref{lem:C_w}. 
	Thus,  we obtain
	\[
	\widehat{\val}(x) = \lim_{n \to \infty} \frac{1}{\abs{W_1} + \dots + \abs{W_n}} \int_{z_0}^{z_n}  j \eta_{x', x}.
	\]
\end{proof}


\section{Elementary units for badly approximable numbers} \label{sec:elem_unit}


In this section, we consider an analog of elementary units for badly approximable numbers and prove \cref{thm:main:extension} \ref{item:thm:main4} and \ref{item:thm:main5}.

\begin{thm}[\cref{thm:main:extension} \ref{item:thm:main4} and \ref{item:thm:main5}] \label{thm:2log}
	Let $ x $ be a badly approximable number and $ W_1, W_2, \dots $ be an infinite sequence of even words such that $ x= [W_1 W_2 \cdots] $ and the set $ \{ W_1, W_2, \dots  \} $ is finite.
	Let $ c_n \in \Z_{>0} $ be the denominator of a rational number $ [W_1 \cdots W_n] $ and
	\[
	L := \lim_{n \to \infty} \frac{\abs{W_1} + \dots + \abs{W_n}}{n}.
	\]
	Then the limit
	\[
	\veps_x := \lim_{n \to \infty} c_n^{1/Ln}
	\]
	converges if and only if $ \widehat{1}(x) $ converges.
	In that case, it holds $ \widehat{1}(x) = 2 \log \veps_x $.		
\end{thm}

Here we call $ \veps_x $ the elementary units for each badly approximable number $ x $.
Before giving a proof, we state a remark.

\begin{rem}
	In the situation in \cref{thm:2log}, we have
	$ \veps(W_1, W_2, \dots ) \ge (3 + \sqrt{5})/2 $.
\end{rem}

\begin{proof}[Proof of Theorem $ \ref{thm:2log} $]
	Let
	\[
	\gamma_n := \gamma_{W_1 \cdots W_n}^{}
	= \pmat{* & * \\ c_n & d_n}.
	\]
	Here $ c_n $ is a denominator of the rational number
	$ \gamma_n(\infty) = [W_1 \cdots W_n] $.
	Pick any point $ z_0 \in \bbH $. 
	Then we have
	\begin{align}
		\widetilde{1}(x)
		= \lim_{n \to \infty} \frac{1}{n}
		\int_{z_0}^{\gamma_n^{} z_0} \frac{-1}{z - x} dz 
		= \lim_{n \to \infty} \frac{1}{n}
		\biggl[ - \log(z - x) \biggr]_{z_0}^{\gamma_n^{} z_0}.
	\end{align}
	Since $ \gamma_n z_0 $ converges to $ x $ as $ n \to \infty $ and the argument of $ \log $ is bounded, this limit value is equal to
	\begin{align}
		\lim_{n \to \infty} \frac{-1}{n} \log \abs{\gamma_n z_0 - x} 
		= \lim_{n \to \infty} \frac{-1}{n}	\log 
		\abs{\frac{z_0 - \gamma_n^{-1} x}
			{(c_n z_0 + d_n) (c_n \gamma_n^{-1} x + d_n)}}.
	\end{align}
	This is equal to
	\[
	\lim_{n \to \infty} \frac{1}{n} \log
	\abs{(c_n z_0 + d_n) (c_n \gamma_n^{-1} x + d_n)} 
	\]
	since $ x $ is badly approximable and $ \{ \gamma_n^{-1} x \mid n \in \Z_{>0} \} $ is bounded by \cref{lem:C_w}.
	Since $ \gamma_n^{-1} \infty = - d_n / c_n $ is bounded by \cref{lem:C_w}, 
	\[
	\left\{ \frac{c_n z + d_n}{c_n \gamma_n^{-1} x + d_n} \right\}_{ n \in \Z_{>0} }
	\]
	is bounded for any point $ z \in \bbH \cup \R $.
	Thus,  we obtain
	\[
	\widehat{1}(x)
	= \lim_{n \to \infty} \frac{2}{n} \log \abs{c_n z + d_n}.
	\]
	By substituting $ z = 0 $, we have
	\[
	\widehat{1}(x)
	= \lim_{n \to \infty} \frac{2}{n} \log d_n.
	\]
	Since $ d_n / c_n $ is bounded, we obtain
	\[
	\widehat{1}(x)
	= \lim_{n \to \infty} \frac{2}{n} \log c_n
	= 2 \log \veps_x.
	\]
\end{proof}

\begin{ex}
	For the case when $ x = \phi := (1 + \sqrt{5})/2 =[1, 1, \dots] $, since
	\[
	[\underbrace{1, \dots, 1}_{2n}] = \frac{F_{2n+1}}{F_{2n}}
	\]
	where $ F_n $ denotes the $ n $-th Fibonacci number, we have
	\begin{align}
		\veps_\phi 
		&= \lim_{n \to \infty} \sqrt[n]{F_{2n}} 
		= \lim_{n \to \infty} \sqrt[n]{\frac{1}{\sqrt{5}} \left( \phi^{2n} - (-\phi)^{2n} \right)} 
		= \phi^2
		= \frac{3 + \sqrt{5}}{2}.
	\end{align}
\end{ex}


\section{An explicit computation of values of extended $ \val $ function} \label{sec:explicit_computation}


In this section, we prove \cref{thm:main:realize_lim}.
The most important point of the proof below is the following lemma which is based on the proof of~\cite[Theorem 4.3]{BI} and is a generalization of~\cite[Lemma 3.3]{Mura}.

\begin{lem}[Repetition frequency estimation] \label{lem:BI_eval_integral}
	Let $\{ a_{1, n} \}_{n=1}^{\infty}$, $ \dots $, $ \{ a_{k, n} \}_{n=1}^{\infty}$ be sequences in $ \Z_{\ge 0} $ such that $ a_{i, n} \to \infty $ and $ 2^{-n} a_{i, n} \to 0 $ as $ n \to \infty $.
	Let $ V $ be a non-empty even word and put $ v := [\overline{V}] $.
	Let $ V_n $ and $ W_n $ be non-empty even words such that $ V_n $ and $ V_{n+1} $, $ W_n $ and $ W_{n+1} $ coincide in the first $ n $ terms respectively.
	Let $\{ x_{n} \}_{n=1}^{\infty} $ and $ \{ x'_{n} \}_{n=1}^{\infty} $ be sequences of real numbers whose continued fraction expansions are
	\[
	x_{n} = [V^{a_n} V_n \cdots], \quad
	-x'_{n} = [0, W_n \cdots]
	\]
	respectively.
	Then 
	\[
	I_n := 
	\int_{z_0}^{\gamma_{V}^{a_n} z_0} j \eta_{x'_n, x_n}^{}
	- a_n N(V) \widehat{\val} (v)
	\]
	is a Cauchy sequence. 
	In particular, we have
	\[
	\int_{z_0}^{\gamma_{V}^{a_n} z_0} j \eta_{x'_n, x_n}^{}
	= a_n N(V) \widehat{\val} (v) + O(1) \quad
	\text{ as } n \to \infty.
	\]
\end{lem}

\begin{proof}
	Take positive integers $ m < n $ with $ a_m < a_n $. 
	We have
	\[
	I_n - I_m =
	\int_{z_0}^{\gamma_{V}^{} z_0} j \left( 
	\sum_{0 \le i < a_n} 
	\eta_{\gamma_{V}^{i-a_n} x'_n, \gamma_{V}^{i-a_n} x_n}^{} 
	- \sum_{0 \le i < a_m} 
	\eta_{\gamma_{V}^{i-a_m} x'_m, \gamma_{V}^{i-a_m} x_m}^{} 
	+ (a_n - a_m) \eta_{v', v}
	\right).
	\]
	Let
	\begin{align}
		S_1(m, n; z) &:=
		\sum_{0 \le i < a_m} \left( 
		\frac{1}{z - \gamma_{V}^{i-a_n} x_n} 
		- \frac{1}{z - \gamma_{V}^{i-a_m} x_m} 
		\right), \\
		S'_1(m, n; z) &:=
		\sum_{0 \le i < a_m} \left( 
		\frac{1}{z - \gamma_{V}^{i-a_n} x'_n} 
		- \frac{1}{z - \gamma_{V}^{i-a_m} x'_m} 
		\right), \\
		S_2(m, n; z) &:=
		\sum_{a_m \le i < a_n} \left( 
		\frac{1}{z - \gamma_{V}^{i-a_n} x_n} 
		- \frac{1}{z - v} 
		\right), \\
		S'_2(m, n; z) &:=
		\sum_{a_m \le i < a_n} \left( 
		\frac{1}{z - \gamma_{V}^{i-a_n} x'_n} 
		- \frac{1}{z - v'} 
		\right).
	\end{align}
	Then we can write
	\[
	I_n - I_m =
	\int_{z_0}^{\gamma_{V}^{} z_0} j \left( 
	- S_1(m, n; z) + S'_1(m, n; z) + S_2(m, n; z) - S'_2(m, n; z)
	\right) dz.
	\]
	For each $ 0 \le i < a_m $, since
	\begin{alignat}{2}
		\gamma_{V}^{i-a_n} x_n &= [V^i V_n \cdots], \quad &
		\gamma_{V}^{i-a_m} x_m &= [V^i V_m \cdots], \\
		\gamma_{V}^{i-a_n} x'_n &= -[0, V'^i W_n \cdots], \quad &
		\gamma_{V}^{i-a_m} x'_m &= -[0, V'^i W_m \cdots],
	\end{alignat}
	we have
	\[
	\abs{\gamma_{V}^{i-a_n} x_n - \gamma_{V}^{i-a_m} x_m}
	\le 2^{2- m}, \quad
	\abs{\gamma_{V}^{i-a_n} x'_n - \gamma_{V}^{i-a_m} x'_m}
	\le 2^{2- m}.
	\]
	For each $ a_m \le i < a_n $, since
	\[
	\gamma_{V}^{i-a_n} x_n = [V^i V_n \cdots], \quad
	\gamma_{V}^{i-a_n} x'_n = -[0, V'^i W_n \cdots],
	\]
	we have
	\[
	\abs{\gamma_{V}^{i-a_n} x_n - v}
	\le 2^{2- a_m}, \quad
	\abs{\gamma_{V}^{i-a_n} x'_n -v'}
	\le 2^{2- a_m}.
	\]
	By \cref{lem:triv_eval_integral}, we obtain
	\[
	\abs{I_n - I_m} 
	= \sum_{0 \le i < a_m} O(2^{-m})
	+ \sum_{a_m \le i < a_n} O(2^{-a_m})
	= O(a_m 2^{-m}) + O((a_n - a_m)2^{-a_m}).
	\]
	Since this converges to 0 as $ m \to \infty $ by the assumption, $ \{ I_n \} $ is a Cauchy sequence. 
\end{proof}

\begin{proof}[Proof of $ \cref{thm:main:realize_lim} $]
	Keep the notation in \cref{thm:main:realize_lim} and let 
	\[
	L := \lim_{n \to \infty} \frac{\abs{U_1} + \cdots + \abs{U_n}}{A_n}.
	\]
	It suffices to show that
	\[
	L \widehat{\val}(x)
	= a_1 \widehat{\val}(W_1) + \dots + a_k \widehat{\val}(W_k).
	\]
	Take any real number $ x' < -1 $ and for positive integers $ n $ and $ 1 \le i \le k $, let
	\begin{align}
		x_n &:= \gamma_{U_1 \cdots U_{n-1}}^{-1} x
		=[U_n U_{n+1} \cdots], \\
		x'_n &:= \gamma_{U_1 \cdots U_{n-1}}^{-1} x'
		= -[0, U'_{n-1} \cdots U'_1, \delta_0 x'], \\
		x_{i, n} &:= \gamma_{V_1 W_1^{a_{1, n}} V_2 W_2^{a_{2, n}} \cdots V_{i-1} W_{i-1}^{a_{i-1, n}}}^{-1} x_n
		=[V_{i} W_{i}^{a_{i, n}} \cdots V_{k} W_{k}^{a_{k, n}} U_{n+1} U_{n+2} \cdots], \\
		x'_{i, n} &:= \gamma_{V_1 W_1^{a_{1, n}} V_2 W_2^{a_{2, n}} \cdots V_{i-1} W_{i-1}^{a_{i-1, n}}}^{-1} x_n
		= -[0, {W'}_{i-1}^{a_{i-1, n}} V'_{i-1} \cdots {W'}_1^{a_{1, n}} V'_1   
		U'_{n-1} \cdots U'_1, \delta_0 x'].
	\end{align}
	Since $ x' < -1 $, the far right-hand sides in the second and fourth rows are continued fraction expansions.
	Then we have
	\begin{align}
		L \widehat{\val}(x)
		&=  \lim_{ n \to \infty } \frac{1}{A_n}
		\sum_{1 \le m \le n} 
		\int_{\gamma_{U_1 \cdots U_{m-1}}^{} z_0}
		^{\gamma_{U_1 \cdots U_m}^{} z_0} 
		j \eta_{x', x}^{} \\
		&=  \lim_{ n \to \infty } \frac{1}{A_n}
		\sum_{1 \le m \le n} 
		\int_{z_0}^{\gamma_{U_m}^{} z_0} 
		j \eta_{x'_m, x_m}^{} \\
		&=  \lim_{ n \to \infty } \frac{1}{A_n} 
		\sum_{1 \le m \le n} 
		\sum_{1 \le i \le k} 
		\left(  
		\int_{z_0}^{\gamma_{V_{i}}^{} z_0} 
		j \eta_{x'_{i, m}, x_{i, m}}^{}
		+ \int_{z_0}^{W_{i}^{a_{i, m}} z_0} 
		j \eta_{\gamma_{V_{i}}^{-1} x'_{i, m}, 
			\gamma_{V_{i}}^{-1} x_{i, m}}^{}
		\right).
	\end{align}
	Since we have
	\[
	\abs{ x_{i, m} - \gamma_{V_{i}}^{} w_i } \le 2^{2- a_{i, m}}, \,
	\abs{ x'_{i, m} - \gamma_{w'_{i-1}} } \le 2^{2- a_{i-1, m}}
	\]
	by continued fraction expansion, it follows from \cref{lem:triv_eval_integral} that
	\[
	\int_{z_0}^{\gamma_{V_{i}}^{} z_0} 
	j \left( \eta_{x'_{i, m}, x_{i, m}}^{} - \eta_{w'_{i-1}, \gamma_{V_{i}}^{} w_i}^{} \right)
	= O(2^{- a_{i, m}} + 2^{- a_{i-1, m}}).
	\]
	Since
	\begin{align}
		\gamma_{V_{i}}^{-1} x_{i, m} &= 
		[W_{i}^{a_{i, m}} \cdots V_{k} W_{k}^{a_{k, m}} U_{n+1} U_{n+2} \cdots], \\
		\gamma_{V_{i}}^{-1} x'_{i, m} &= 
		-[0, V'_{i} {W'}_{i-1}^{a_{i-1, m}} V'_{i-1} \cdots {W'}_1^{a_{1, m}} V'_1 U'_{m-1} \cdots U'_1, \delta_0 x'],
	\end{align}
	we have
	\[
	\int_{z_0}^{W_{i}^{a_{i, m}} z_0} 
	j \eta_{\gamma_{V_{i}}^{-1} x'_{i, m}, 
		\gamma_{V_{i}}^{-1} x_{i, m}}^{}
	= a_{i, m} N(W_i) \widehat{\val}(w_i) + O(1)	
	\]
	by \cref{lem:BI_eval_integral}.
	Thus, we obtain
	\begin{align}
		L \widehat{\val}(x)
		&=  \lim_{ n \to \infty } \frac{1}{A_n} 
		\sum_{1 \le m \le n} 
		\sum_{1 \le i \le k} 
		a_{i, m} N(W_i) \widehat{\val}(w_i) \\
		&= \sum_{1 \le i \le k} 
		a_i N(W_i) \widehat{\val}(w_i).
	\end{align}
\end{proof}


\section{The value of $ \val $ function at the Thue--Morse word} \label{sec:Thue--Morse}


In this section, we calculate the value of $ \val $ function at the Thue--Morse word and prove \cref{thm:Thue--Morse}.

We fix two even words $ V $ and $ W $.
Let $ h \colon \{ V, W \}^* \to \{ V, W \}^* $ be a monoid homomorphism such that $ h(V) := VW, h(W) := WV $ and $ V_h := \lim_{n \to \infty} h^n(V) $ be the Thue--Morse word.
For a word $ U = U_1 \cdots U_n \in \{ V, W \}^* $ with $ U_1, \dots, U_n \in \{ V, W \} $, let $ U' := U_n \cdots U_1 $.

To begin with, we give a simple representation of the Thue--Morse word.

\begin{lem} \label{lem:Thue--Morse}
	Let $ \tau \colon \{ V, W \}^* \to \{ V, W \}^* $ be a monoid homomorphism such that $ \tau(V) := W $ and $ \tau(W) := V $.
	Then the following statements hold.
	\begin{enumerate}
		\item \label{item:lem:Thue--Morse1} For any positive integer $ n $, we have $ h^n(V) = h^{n-1}(V) \tau h^{n-1}(V) $.
		\item \label{item:lem:Thue--Morse2} $ h^{2n}(V)' = h^{2n}(V) $.
	\end{enumerate}
\end{lem}

\begin{proof}
	The first claim \ref{item:lem:Thue--Morse1} follows from the facts that $ h \circ \tau = \tau \circ h $ and $ h^2(U) = h(U) \tau h(U) $ for a word $ U \in \{ V, W \}^* $ such that $ h(U) = U \tau(U) $.
	The last claim \ref{item:lem:Thue--Morse2} follows from the fact that $ h^2(U)' = h^2(U) $ for a word $ U \in \{ V, W \}^* $ such that $ h(U) = U \tau(U) $ and $ U' =U $.
\end{proof}

\begin{proof}[Proof of $ \cref{thm:Thue--Morse} $]
	Let $ V_n := h^n(V) $.
	Since $ [V_n]' = -[0, V_n] $ by \cref{lem:Thue--Morse} \ref{item:lem:Thue--Morse2}, we have
	\[
	\widehat{\val}([V_h]) = \lim_{n \to \infty} \frac{1}{2^{2n}} 
	\int_{z_0}^{\gamma_{V_n}^{} z_0} j \eta_{-[0, V_h], V_h}^{}, \quad
	\widehat{\val}([V_n]) = 
	\int_{z_0}^{\gamma_{V_n}^{} z_0} j \eta_{-[0, V_n], V_n}^{}.
	\]
	Here the first $ 2^{2n} $ terms of $ V_h $ are $ V_n $ by \cref{lem:Thue--Morse} \ref{item:lem:Thue--Morse1}.
	Thus,  we obtain the claim. 
\end{proof}


\bibliographystyle{alpha}
\bibliography{myrefs_val_extension}

\end{document}